\documentclass[reqno]{amsart}

\usepackage{amsmath}
\usepackage{amssymb}
\usepackage{amsfonts}
\usepackage{graphicx}
\usepackage{amsthm}
\usepackage{enumerate}
\usepackage{lscape}
\usepackage{dsfont}
\usepackage{color}
\usepackage{mathtools}

\usepackage{setspace}
\newtheorem{theor}{Theorem}[section]

\newtheorem{remar}[theor]{Remark}

\newtheorem{corol}[theor]{Corollary}
\newtheorem{lem}[theor]{Lemma}
\newtheorem{guess}{Conjecture}

\newcommand{\Z}{\mathbb{Z}}
\newcommand{\CP}{\mathds{C}\mathrm{P}}

\newcommand{\dd}{\mathrm{D}}

\newcommand{\R}{\mathbb{R}}

\newcommand{\C}{\mathbb{C}}

\newcommand{\N}{\mathbb{N}}

\newcommand{\K}{K\"{a}hler\ }
\newcommand{\zbar}{\overline{z}}
\renewcommand{\div}{\textrm{div}}
\newcommand{\Ric}{\textrm{Ric}}

\begin{document}
\title[]{Two conjectures on Ricci-flat K\"{a}hler metrics}

\author{Andrea Loi, Filippo Salis, Fabio Zuddas}
\address{Dipartimento di Matematica e Informatica, Universit\`a di Cagliari\\Via Ospedale 72, 09124 Cagliari (Italy)}
\email{loi@unica.it, filippo.salis@gmail.com, fabio.zuddas@unica.it}

\thanks{The first author was  supported by Prin 2015 -- Real and Complex Manifolds; Geometry, Topology and Harmonic Analysis -- Italy,  by INdAM. GNSAGA - Gruppo Nazionale per le Strutture Algebriche, Geometriche e le loro Applicazioni and also by  GESTA - Funded by Fondazione di Sardegna and Regione Autonoma della Sardegna.}

\begin{abstract}
We propose  two conjectures about Ricci-flat \K  metrics:

\noindent
Conjecture 1:
{\em A Ricci-flat projectively induced metric is flat.}

\noindent
Conjecture 2:
{\em A Ricci-flat metric on an $n$-dimensional complex manifold such that  the $a_{n+1}$ coefficient of the TYZ expansion vanishes is flat.}

We verify Conjecture 1 (see Theorem \ref{theomain}) under the assumptions  that the metric  is radial and  stable-projectively induced and Conjecture 2 (see Theorem \ref{theomain2}) for complex surfaces whose metric is either radial 
or complete and ALE.
We end the paper  by showing, by means of the Simanca metric, that the assumption of Ricci-flatness in Conjecture \ref{guess1} and in Theorem \ref{theomain2} cannot be weakened to scalar-flatness (see Theorem \ref{theomain3}).
\end{abstract}
 
\maketitle
\tableofcontents

\section{Introduction}
An interesting open question in K\"ahler geometry is concerned with the characterization of K\"ahler-Einstein projectively induced metrics. Here a K\"ahler metric $g$ on a complex manifold $M$ is said to be \emph{projectively induced} if there exists a K\"ahler (isometric and holomorphic) immersion of $(M, g)$ into the complex projective space $(\C P^N, g_{FS})$, $N \leq +\infty$,  endowed with the Fubini--Study metric $g_{FS}$, namely the metric whose associated \K form is given in homogeneous coordinates  by  $\omega_{FS}=\frac{i}{2}\partial\bar\partial\log (|Z_0|^2+\cdots +|Z_N|^2)$.

When $N$ is finite, the only known examples of complete K\"ahler-Einstein projectively induced metrics are compact and  homogeneous
and it is still an open problem to show that these are the only possibilities (see \cite{Ch}, \cite{Sm}, \cite{Ha}, \cite{AL2}, \cite{ALZ}, \cite{Ta}).
 Indeed one can prove (see, e.g. \cite{DHL}) that, given a compact simply connected homogeneous K\"ahler (Einstein) manifold   $M$  with integral K\"ahler form $\omega$, then the Kodaira map $k: M \rightarrow \C P^N$  suitably normalized is a K\"ahler immersion. 
One can see \cite{DHL} that the assumptions of simply-connectedness and compactness of $M$ and the integrality of $\omega$ are necessary (this excludes for example the compact flat torus to be projectively induced). Moreover, if  the \K form  $\omega$ of a compact 
and  simply-connected homogeneous  \K manifold  $(M ,g)$ is integral then there exists a positive integer $k$ such that $kg$ is projectively induced (see \cite{Ta} for a proof based on semisimple Lie groups and Dynkin diagrams). This last  assertion is valid also for noncompact  simply-connected homogeneous K\"ahler (Einstein) manifolds by considering instead of the Kodaira map the coherent states map (coming from the theory of geometric quantization) and  by allowing the ambient space to be infinite dimensional, namely by considering \K immersion into $\CP^{\infty}$ (see \cite{loimossaber}).
Nevertheless  there exist complete and nonhomogeneous projectively induced   K\"ahler-Einstein
metrics on Cartan-Hartogs domains  with {\em negative} (constant) scalar curvature  (see \cite{LZ}).

Notice that in the noncompact case, due for example to the fact that $\lambda \omega$ is always integral provided $M$ is contractible, the structure of the set of the positive real numbers $\lambda \in \R^+$ for which $\lambda g$ is projectively induced is in general less trivial than in the compact case (where it is always discrete). For example, in the noncompact symmetric case one has the following (see also \cite{loimossahom} for the more general case of bounded homogeneous domains):

\vskip 0.3cm
\noindent
{\bf Theorem A} (Theorem 2 in \cite{LZ})
{\em Let $\Omega$ be an irreducible bounded symmetric domain endowed with its
Bergman metric $g_B$. Then there exist a positive real number $a$ and an integer $k$ (both depending on $\Omega$) such that $(
\Omega, \lambda g_B)$ admits an equivariant K\"ahler immersion into $\C P^{\infty}$ if
and only if $\lambda$
 belongs to the set
$$\left\{ a, 2a, \dots, ka \right\} \cup (ka, +\infty).$$}

From this theorem it follows that the only irreducible bounded symmetric domain where $\lambda g_B$ is projectively induced for all $\lambda > 0$ is the complex hyperbolic space. 
More generally, for a homogeneous bounded domain $(\Omega, g)$ we have that $\lambda g$ is projectively induced for all $\lambda > 0$ if and only if $(\Omega, g) = \C H^{n_1}_{\lambda_1} \times \cdots \times \C H^{n_r}_{\lambda_r}$, where $\C H^{n_j}_{\lambda_j} = (\C H^{n_j}, \lambda_j g_{hyp})$ (\cite{DHL}, Theorem 4).

Inspired by these results, we give the following definition:

{\em A \K  metric $g$  is said to be  {\em stable-projectively induced} if there exists $\epsilon >0$ such that $\lambda g$ is projectively induced for all $\lambda \in (1-\epsilon, 1+\epsilon)$. A \K metric is said to be  {\em unstable} if it is not stable-projectively induced.}

Obviously a \K metric on a compact complex manifold is always unstable and Theorem A shows 
that there exists metrics $g$ which are projectively induced and unstable and which  become stable-projectively induced by multiplying them for a suitable constant. Notice also that the flat metric $g_0$ on the complex Euclidean space $\C^n$ is stable-projectively induced by the map $\Psi: \C^n \rightarrow \C P^{\infty}, \ \Psi(z) = \left( \dots, \sqrt{\frac{1}{m_j!}} z^{m^j} , \dots \right)$ (see \cite{Cal}). Consequently, many examples of stable-projectively induced metrics can be constructed on those complex manifolds $M$ which admit 
a holomorphic immersion into $\C^n$ (e.g. Stein manifolds) by simply taking the restriction of the flat metric $g_0$ to  $M$.

For the case of Ricci-flat metrics, namely  K\"ahler-Einstein metrics with Einstein constant zero,  D. Hulin  \cite{hulinlambda} proves that a compact K\"ahler-Einstein manifold K\"ahler immersed into $\C P^N$ has positive scalar curvature. This
result implies for example that a Calabi-Yau manifold
does not admit a K\"ahler immersion into $\C P^N$. 
On the other hand there are many interesting examples of Ricci-flat metrics on noncompact complex manifolds, for example the celebrated  {\it Taub-NUT metric}, defined as the family (constructed by C. Le Brun) of complete K\"ahler forms on $\C^2$ given by $\omega_m = \frac{i}{2} \partial \bar \partial \Phi_m$, for $m \geq 0$, where $\Phi_m(u, v) = u^2 + v^2 + m(u^4 + v^4)$ and $u$ and $v$ are implicitly defined by $|z_1| = e^{m(u^2+v^2)}u$, $|z_2| = e^{m(v^2-u^2)}v$ (notice that for $m=0$ one gets the flat metric on $\C^2$). Then one can prove \cite{taubnut} that for $m > \frac{1}{2}$ there does not exist a K\"ahler immersion of $(\C^2, \omega_m)$ into $\C P^{\infty}$.

Thus, we believe the validity of the following conjecture:

\begin{guess}\label{guess1}
A Ricci-flat projectively induced metric is flat.
\end{guess}

In this paper we verify Conjecture \ref{guess1} under the assumptions that the metric involved  is stable-projectively induced and restricting ourselves to {\it radial} K\"ahler metrics, i.e. those admitting a  K\"ahler potential $\Phi$ which depends only on the sum $|z|^2 = |z_1|^2 + \cdots + |z_n|^2$ of the local coordinates' moduli. Our first result is then the following:

\begin{theor}\label{theomain}
The only Ricci-flat, stable-projectively induced and  radial K\"ahler metric is the flat one.
\end{theor}

Notice that without assuming the Ricci-flatness the thesis of the previous theorem does not hold. For  example  the radial non-flat \K metric $g=\frac{i}{2}\partial\bar\partial (|z|^2+|z|^4)$ on $\C$
is stable-projectively induced being  the pull-back of the flat metric on $\C^2$ via the embedding $z\mapsto (z, z^2)$.

The requirement that a \K metric is projectively induced is a somehow strong assumption.
Thus it is natural to try to approximate a \K metric $g$ on a complex manifold $M$ with  projectively induced ones.
In the last two decades a lot of work has been done in this direction both in the noncompact and compact case.
Roughly speaking, if the \K form $\omega$ associated to $g$ is integral, then for every positive integer $m$ one can construct  a  holomorphic map $\varphi_m: M\rightarrow \CP^{N_m}$ into an 
$N_m$-dimensional ($N_m\leq \infty$) complex projective space such that 
$$\lim_{m\rightarrow\infty}\frac{1}{m}\varphi_m^*g_{FS}=g.$$
More precisely,  under suitable assumptions (automatically satisfied in the compact case) (see, e.g. \cite{arezzoloi})
there exists a smooth function $\epsilon_{mg}$ on $M$, depending on $m$ and on the metric $g$, such that 
$$\varphi_m^*\omega_{FS}=m\omega+\frac{i}{2}\partial\bar\partial\log\epsilon_{m g}$$
and admitting the so called {\em Tian--Yau--Zelditch expansion} (TYZ in the sequel)
\begin{equation}\label{TYZ}
\epsilon_{m g}(x)\sim \sum_{j=0}^\infty a_j(x)m^{n-j},
\end{equation}
where $a_0(x)=1$ and $a_j(x)$, $j=1,\dots$ are smooth functions on $M$ depending on the curvature 
and its covariant derivatives at $x$ of the metric $g$ (see \cite{ze} for details).
In particular, Z. Lu \cite{lu} computed the first three coeffcients: 
\begin{equation}\label{coefflu}
\left\{\begin{array}
{l}
a_1(x)=\frac{1}{2}\rho\\
a_2(x)=\frac{1}{3}\Delta\rho
+\frac{1}{24}(|R|^2-4|{\rm Ric} |^2+3\rho ^2)\\
a_3(x)=\frac{1}{8}\Delta\Delta\rho +
\frac{1}{24}{\rm div}{\rm div} (R, {\rm Ric})-
\frac{1}{6}{\rm div}{\rm div} (\rho{\rm Ric})+\\
\qquad\quad\ \ +\frac{1}{48}\Delta (|R|^2-4|{\rm Ric} |^2+8\rho ^2)+
\frac{1}{48}\rho(\rho ^2- 4|{\rm Ric} |^2+ |R|^2)+\\
\qquad\quad\ \ +\frac{1}{24}(\sigma_3 ({\rm Ric})- {\rm Ric} (R, R)-R({\rm Ric} ,{\rm Ric})),
\end{array}\right.
\end{equation}
where 
$\rho$, $R$, $\Ric$ denote respectively the scalar curvature,
the curvature tensor and the Ricci tensor of $(M, g)$,
and we are using  the following notations (in local coordinates $z_1, \dots , z_n$):
\begin{equation}\begin{array}{l}\label{values}
|D^{'}\rho|^2=  g^{j \bar i} \frac{\partial \rho}{\partial z_i} \frac{\partial \rho}{\partial \bar z_j},\\
|D^{'}\Ric|^2= g^{\alpha \bar i} g^{j \bar \beta} g^{\gamma \bar k}  \Ric_{i\bar j, k} \overline{\Ric_{\alpha \bar \beta, \gamma}} ,\\
|D^{'}R|^2 = g^{\alpha \bar i}g^{j \bar \beta}g^{\gamma \bar k}g^{l \bar \delta}g^{\epsilon \bar p} R_{i \bar j k \bar l, p} \overline{R_{\alpha \bar \beta \gamma \bar \delta, \epsilon}} ,\\
\div\div (\rho \Ric)=2|D^{'}\rho|^2+
g^{\beta \bar i} g^{j \bar \alpha} \Ric_{i\bar j}\frac{\partial^2 \rho}{\partial z_{\alpha} \partial \bar z_{\beta}}
+\rho\Delta\rho,\\
\div\div (R, \Ric)=
-g^{\beta \bar i} g^{j \bar \alpha} \Ric_{i\bar j}\frac{\partial^2 \rho}{\partial z_{\alpha} \partial \bar z_{\beta}}
-2|D^{'}\Ric|^2+\\
\qquad\qquad\qquad\quad\ \ + g^{\alpha \bar i} g^{j \bar \beta} g^{\gamma \bar k} g^{l \bar \delta} R_{i\bar j,k\bar l}R_{\beta \bar \alpha \delta \bar \gamma}-
R(\Ric, \Ric)-\sigma_3(\Ric),\\
R(\Ric, \Ric)= g^{\alpha \bar i}g^{j \bar \beta}g^{\gamma \bar k}g^{l \bar \delta} R_{i\bar jk\bar l}\Ric_{\beta \bar \alpha}\Ric_{\delta \bar \gamma},\\
\Ric(R, R) = g^{\alpha \bar i}g^{j \bar \beta}g^{\gamma \bar k}g^{\delta \bar p}g^{q \bar \epsilon} \Ric_{i\bar j}R_{\beta \bar \gamma p \bar q}R_{k \bar \alpha \epsilon \bar \delta},\\
\sigma_3 (\Ric)=g^{\delta \bar i}g^{j \bar \alpha}g^{\beta \bar \gamma} \Ric_{i\bar j}\Ric_{\alpha \bar \beta}\Ric_{\gamma \bar \delta},\\
\end{array}\end{equation}
where the $g^{j \bar i}$'s denote the entries of the inverse matrix of the metric (i.e. $g_{k \bar i} g^{j \bar i} = \delta_{kj}$),  \lq\lq\  ,p'' represents the covariant derivative
in the direction $\frac{\partial}{\partial z_p}$ and we are using the summation convention for repeated indices.

 The reader is also referred to  
\cite{loianal} and \cite{loismooth}  for a  recursive formula  for the coefficients $a_j$'s  and  an alternative computation of  $a_j$ for $j\leq 3$ using Calabi's diastasis function (see also  \cite{xu1}  for a graph-theoretic
interpretation of this recursive formula).

Due to Donaldson's work (cfr. \cite{donaldson, do2, arezzoloi}) in the compact case and respectively to the theory of
quantization  in the noncompact case (see, e.g. \cite{Ber1, cgr3, cgr4}),
 it is  natural to study metrics with the coefficients  of  the  TYZ expansion being prescribed.
In this regard  Z. Lu and G. Tian   \cite{lutian}  (see also \cite{engliszhang} and \cite{LoiArezzo} for the symmetric and homogenous case
respectively) prove that the PDEs $a_j=f$ ($j\geq 2$ and $f$ a smooth function on $M$) are elliptic and that if the logterm
of the Bergman and  Szeg\"{o} kernel  of the unit disk bundle over $M$ vanishes then  $a_k=0$, for  $k>n$ ($n$ being the complex dimension of $M$).  
The study of these PDEs makes sense regardless of the existence of  a TYZ expansion
and so given any K\"ahler manifold $(M, g)$ it makes sense to call the   $a_j$'s  the {\em coefficients associated to metric $g$}.
In the noncompact case in \cite{taubnut}
one can find a characterization of the flat metric as a Taub-Nut metric with $a_3=0$ while Feng and Tu \cite{fengtu} solve a conjecture formulated in \cite{zedda} by showing that  the complex hyperbolic space is the only Cartan-Hartogs domain where the coefficient  $a_2$ is constant. In a recent paper \cite{LZhs} the first author together with M. Zedda prove that a locally hermitian symmetric space with vanishing $a_1$
and $a_2$ is flat.

In this paper we address the following:

\begin{guess}\label{guess2}
A Ricci-flat metric on an $n$-dimensional complex manifold such that $a_{n+1}=0$ is flat.
\end{guess}

In the following theorem, which represents our second result, we verify Conjecture \ref{guess2} for (compact or noncompact) complex surfaces under the assumption that the metric is either  ALE (Asymptotically Locally Euclidean) or radial. 

Roughly speaking, an $n$-dimensional complete Riemannian manifold $(M, g)$ is said to be ALE if there exists a compact subset $K \subset M$ such that $M \setminus K$ is diffeomorphic to the quotient of $\R^n \setminus B_R(0)$ (the ball of  radius $R>0$) by a finite group $G \subset O(n)$, and such that the metric $g$ on this open subset tends to the flat euclidean metric at infinity. 
For the exact definition and construction of ALE \K metrics, which are interesting both from the mathematical and the physical point of view, the reader is referred to the foundational paper \cite{kro} (see also \cite{BKM}, \cite{J}, \cite{NST}, \cite{M}): in this paper we will need just the fact that the norm of the curvature tensor of such metrics vanishes at infinity.

\begin{theor}\label{theomain2}
Let $(M, g)$ be a Ricci-flat \K  surface such that the third coefficient $a_3$ of the TYZ expansion vanishes. 
Assume that  one of the following two conditions holds true:
\begin{itemize}
\item [1.]
$g$ is complete and ALE (asymptotically locally Euclidean);
\item [2.]
$g$ is radial.
\end{itemize}
Then $g$ is flat.
\end{theor}

We end the paper  by showing that the assumption of Ricci-flatness in Conjecture \ref{guess1} and in Theorem \ref{theomain2} cannot be weakened to scalar-flatness. Indeed we prove the following:
\begin{theor}\label{theomain3}
The Simanca metric $g_S$ on the blown-up $\CP^2\sharp\overline{\CP^2}$ of $\CP^2$  at one point  is an ALE complete radial  projectively induced scalar flat (and not Ricci-flat) metric  with vanishing $a_3$.
\end{theor}

The paper is organized as follows. In Section 2 we recall the definition and properties of Calabi's diastasis function, which is the main tool for the proof of our results, and we apply Calabi's theory to radial metrics defined on open domains of $\C^n\setminus \{0\}$ obtaining Lemma
\ref{projind}, a fundamental tool in this paper. Finally, Section 3 and 4 are dedicated to the proofs of Theorem \ref{theomain} and Theorems \ref{theomain2} and \ref{theomain3} respectively.

\section{Radial projectively induced metrics}
In order to prove our theorems we need to recall the definition of Calabi's diastasis function and some of its properties.
Let $(M,g)$ be a K\"ahler manifold with a local \K potential $\Phi$, i.e. such that $\omega = \frac{i}{2} \partial \bar \partial \Phi$, where $\omega$ is the K\"ahler form associated to $g$. A K\"ahler potential is not unique, but it is defined up to an addition of the real part of a holomorphic function. If $g$ (and hence $\Phi$) is assumed to be real analytic, by duplicating the variables $z$ and $\bar z$, $\Phi$ can be complex analytically continued to a function $\hat\Phi$ defined in a neighbourhood $U$ of the diagonal containing $(p,\bar p)\in M\times \bar M$ (here $\bar M$ denotes the manifold conjugated to $M$). 

Then the \emph{diastasis function} $\dd_p(z)$ for $g$ is defined to be the unique K\"ahler potential around $p$ given by
$$D_p(z)=\hat\Phi(z,\bar z)+ \hat \Phi(p,\bar p) -\hat \Phi(z,\bar p)-\hat \Phi(p,\bar z).$$
By shrinking $U$ if necessary we can assume that $D_p$ is defined on $U$.

As shown by the statement of the following lemma, the diastasis turns out to be an important tool to study projectively induced metrics. 

\begin{lem} [Calabi \cite{Cal}]\label{calabi}
Let $(M, g)$ be a \K manifold. 
There exists a neighborhood of a point $p\in M$ that admits a \K immersion into $(\CP^N,g_{FS})$, with $N\leq\infty$, if and only if the metric $g$ is $1$-resolvable at $p$ of rank at most $N$. If $M$ is connected the $1$-resolvability does not depend on the point chosen. Moreover,
if $M$ is simply-connected and $g$ is $1$-resolvable at a point  then there exists a global \K immersion from $(M, g)$ into 
$(\CP^{N}, g_{FS})$.
\end{lem}

\vskip 0.3cm
A \K metric with diastasis $D_p(z)$ is \emph{$1$-resolvable at $p$ of rank $N$} if the matrix $B_{i,j}$, defined by considering the expansion around the point $p$ of the function $e^{D_p(z)}-1=\sum_{m_i,m_j\in \N^n}B_{i,j}(z-p)^{m_i} (\zbar-\bar p)^{m_j}$, is positive semidefinite and its rank is $N$. 
Here, $z^{m_j}$ denotes the monomial in $n$ variables $\prod_{\alpha=1}^n z_\alpha^{m_{\alpha, j}}$ and we arrange every $n$-tuple of nonnegative integers as a sequence $m_j=(m_{1,j},\dots,m_{n,j})$ such that $m_0=(0,\dots,0)$, $|m_j|\leq |m_{j+1}|$ for all positive integer $j$ and all the $m_j$'s with the same $|m_j|$ using lexicographic order.

In particular, we are going to study metrics which admit a \K potential $\Phi:U\rightarrow \R$ that depends only on the sum of the local coordinates'  moduli defined on a domain that does not contain the origin. 
Namely, there exists 
$f:\tilde U\rightarrow \R$, $\tilde U\subset\R^+$,
such that 
\begin{equation}\label{deff}
f (x)=\Phi (z),\  z=(z_1, \dots , z_n),
\end{equation} 
where 
$$\tilde U=\{x=|z|^2=|z_1|^2+\cdots +|z_n|^2 \ | \ z\in U\}.$$
Unlike the case in which the origin is contained in the domain of definition of the diastasis, the matrix $B_{i,j}$ is not diagonal, so it is more difficult to apply Lemma \ref{calabi}
(see, e.g. \cite{diastcigar} for the case on which the origin is contained). The following lemma is the key ingredient for the proof of our results.
\begin{lem}\label{projind}
Let $n\geq 2$ and  $p=(s,0,\mathellipsis ,0)$, with $s\in\R$, $s\neq 0$, be a point of the complex domain $U\subset\C^n\setminus\{0\}$ on which is defined a radial metric $g$  with radial  \K potential $\Phi:U\rightarrow \R$ and corresponding diastasis  $D_p:U\rightarrow \R$.
Let $f:\tilde U \rightarrow \R$ defined by
(\ref{deff}) and, for  $h\in\N$,  let $g_h: \tilde U\rightarrow \R$ given by:
\begin{equation}\label{gh}
g_h(x)=\frac{d^{h}e^{f(x)}}{dx^{h}}e^{-f(x)}.
\end{equation}
Assume that  the entries of the following infinite matrix 				
\begin{equation}\label{matrix}
\Big( det\Big(\frac{1}{i!j!}\frac{\partial^{i+j}(e^{D_p(z)}g_h(|z|^2))}{\partial z_1^{i}\partial \zbar_1^{j}}\Big)_{0\leq i,j\leq l}\Big)_{l,h\in\N}
\end{equation}
are positive when evaluated in $p$. Then  the metric $g$ is $1$-resolvable at $p$ of infinite rank.
\end{lem}

\begin{proof}
Let $z=(z_1, z_2, \mathellipsis ,z_n)=(z_1,z^*)$, let $m_i=(m_{1,i},m_i^*)\in\N^n$ and let $D_p(z)$ be the diastasis function. We observe that if $m_i^*\neq m_j^*\in\N^{n-1}$ then
\begin{equation}\label{diag0}
\frac{\partial^{|m_i|+|m_j|}}{\partial z^{m_i}\partial{\zbar}^{m_j}}(e^{D_p(z)}-1)\big|_p=0.
\end{equation}
In fact, by definition of diastasis, $D_p(z)$ is the the sum of the \K potential $f(|z|^2)$, the constant $f(s^2)$ and the real part of a holomorphic function
which depends only on $z_1$ and which is equal to $-2f(s^2)$ if evaluated in $s$. Therefore 
\begin{equation}\label{Ah}
\frac{\partial^{|m_i|+|m_j|}(e^{D_p(z)}-1)}{\partial z_1^{m_{1,i}}\partial \zbar_1^{m_{1,j}}\partial {z^*}^{m_i^*}\partial{\zbar^*}^{m_j^*}} \Big|_p= \frac{\partial^{|m_i|+m_{1,j}}}{\partial z_1^{m_{1,i}}\partial \zbar_1^{m_{1,j}}\partial {z^*}^{m_i^*}}({z^*}^{m_j^*}e^{D_p(z)-f}\frac{d^{|m_j|}}{dx^{|m_j|}}e^f)\big|_p,
\end{equation}
where $x=|z|^2$. From which we can deduce obviously (\ref{diag0}) and also 
\begin{equation}\label{diag}
\frac{\partial^{|m_j^*|+|m_j^*|}}{\partial {z^*}^{m_j^*}\partial{\zbar^*}^{m_j^*}}(e^{D_p(z)}-1)\big|_p=m_j^*!g_{|m_j^*|}(s^2).
\end{equation}

Now, notice that in order to check if a metric is $1$-resolvable, we are free to change the above arrangement of the multiindices $m_i$'s, because this has just the effect to apply the same permutation to both rows and columns of the matrix $B_{i,j}=\frac{1}{m_i!m_j!}\frac{\partial^{|m_i|+|m_j|}(e^{D_p(z)}-1)}{\partial z^{m_{i}}\partial \zbar^{m_{j}}} \big|_p$ defined by the expansion around the point $p$ of the function $(e^{D_p(z)}-1)$, and then yields a similar matrix, which is positive definite if and only if the original one is.  
In particular, let us change the ordering of the $m_i$'s as follows:
$m_0=(0,\dots,0)$, $|m_j|\leq |m_{j+1}| \text{ for all positive integer } j$,
$\text{if } |m_i|=|m_{j}| \text{ and } m_{1,i}>m_{1,j} \text{ then } i<j$.

With this order, the square submatrix $E_h$ of $B_{i,j}$ relative to multi-indices $m_i$ such that $|m_i|\leq h$ assumes the following form

\begin{equation}\label{blocchiAD}
\left(
\begin{array}{c|c}
A_h & 0 \\
\hline
0 & D_h
\end{array}
\right)
\end{equation}
where $A_h$ is the square matrix relative to multi-indices  $m_i$ such that $|m_i|<h$ or $|m_i|=h$ and $m_{1,i}\neq 0$, while $D_h$ is the matrix relative to multi-indices $m_i$ such that $|m_i|=h$ and $m_{1,i} = 0$.
Indeed, if $|m_i| < h$ and $|m_j|=h, m_{1,j} =0$, then $m_i^* \neq m_j^*$ because, if not, we would have $|m_i| \geq |m_j|$; similarly we clearly have $m_i^* \neq m_j^*$ provided $|m_i| = |m_j| = h$ and $m_{1,i} \neq 0, m_{1,j} = 0$. This, by (\ref{diag0}), explains the null blocks in (\ref{blocchiAD}).
Moreover, it follows again by (\ref{diag0}), combined with the fact that $m_i \neq m_j$, $|m_i| = |m_j| = h$ and $m_{1,i} = m_{1,j} = 0$ imply $m_i^* \neq m_j^*$, that $D_h$ is diagonal (and the entries on the diagonal are described by (\ref{diag})) 
Now, if every matrix $E_h$ is positive definite, namely if for every positive integer $h$ the matrix $A_h$ is positive definite and the entries of $D_h$ are positive, the metric examined is $1$-resolvable at $p$ of infinite rank.

Since we obtain the entries of $D_h$ by multiplying $g_{h}|_p$ for a positive constant, these are positive for every integer $h$ if and only if the entries of the first row ($l=0$) of the matrix (\ref{matrix}), given by $e^{D_p(z)} g_h, \ h= 0,1, \dots$, are positive.

Now we consider the matrix $A_h$ and  we change again the order of the $m_i$'s as follows: 
$|m_j^*| < |m_{j+1}^*| \text{ for all positive integer } j$, \text{if } $|m_i^*|=|m_{j}^*|$ \text{and}  $m_i^*$ \text{precedes}  $m_j^*$ \text{with respect to the lexicographical order or if} $m_i^*=m_j^*$ \text{and} $m_{1,i}<m_{1,j}$ \text{ then } $i<j$.
Then, after the corresponding rows and columns exchanges on $A_h$  and by using  (\ref{diag0}) we obtain a block matrix of the following form:
$$  
\left(
\begin{array}{ccccc}
M_0^h  & 0       & \cdots       &\cdots &0  \\
0      &\ddots   & \ddots       &       &\vdots  \\
\vdots &\ddots   &M_{|m_j^*|}^h &\ddots &\vdots   \\
\vdots &         &\ddots        &\ddots & 0 \\
0      &\cdots	 &\cdots		&0		& M_{h-1}^h\\
\end{array}
\right)
$$
where $M_k^h$ are square matrices whose main diagonal belongs to the main diagonal of the whole matrix and, for the same reason, are themselves block matrices of the same type. By (\ref{Ah}), each block of $M_k^h$ is equal to
$$\Big(\frac{1}{i!j!}\frac{\partial^{i+j}(e^{D_p(z)}g_k)}{\partial z_1^{i}\partial \zbar_1^{j}}\Big)_{0\leq i,j\leq h-k}$$
multiplied by a positive constant. Therefore, by using Sylvester's criterion, if the entries from the second row onwards of the matrix (\ref{matrix}) are positive, $A_h$ is positive definite for every  integer $h$.
\end{proof}

\begin{corol}\label{notprojind}
Under the same assumptions of  Lemma \ref{projind}, if there exists  $x\in\tilde U$ and $h\in\N$ such that  the function given by (\ref{gh}) is negative, namely $g_h(x)<0$, then 
the metric $g$ is not projectively induced.
\end{corol}
\begin{proof}
If follows by combining Lemma \ref{calabi}, Lemma \ref{projind} and the  observation that the entries of the first row of the matrix (\ref{matrix}) are given by  $e^{D_p(z)}g_h(|z|^2)$, $h\in\N$.
\end{proof}

\section{Proof of Theorem \ref{theomain}}
In order to prove Theorem \ref{theomain} we need the following (well-known) classification of the potentials of radial Ricci-flat metrics
(cfr. \cite{calabi}).
\begin{lem}\label{radialcalabi}
Let $U$ be a complex domain  of  $\C^n$  equipped with a radial \K  Ricci-flat metric $g$.
Then there esist $\lambda\in\R^+$ and  $\epsilon=-1, 0, 1$ such that  the function 
$f:\tilde U\rightarrow \R$ defined by (\ref{deff})
has the following expression 
\begin{equation}\label{equationfint}
f(x)=\lambda \int (\epsilon x^{-n} + 1)^{\frac{1}{n}} dx.
\end{equation}
\end{lem}
\begin{proof}
The \K form $\omega$ associated to $g$ reads as:
$$\omega =\frac{i}{2}\sum_{\alpha,\bar \beta}g_{\alpha\bar \beta}dz_\alpha\wedge d\bar z_\beta= \frac{i}{2}\partial \bar \partial \Phi,$$
Sinces $g$  Ricci-flat its Ricci form vanishes, namely
\begin{equation}\label{Ric0}
\rho = -i \partial \bar \partial \log \det(g) = 0
\end{equation}
where
$$g=(g_{\alpha\bar\beta}) = \left( \begin{array}{cccc}
f' + f'' |z_1|^2 & f'' \bar z_1 z_2 & \dots & f'' \bar z_1 z_n \\
f'' \bar z_2 z_1 & f' + f'' |z_2|^2 & \dots & f'' \bar z_2 z_n \\
& & \dots  & \\
f'' \bar z_n z_1 &  f'' \bar z_n z_2 & \dots & f' + f'' |z_n|^2 \\
\end{array} \right).$$

Thus, one easily sees that
$$\det(g) = (f')^{n-1}(f'+ f'' x),\  x=|z|^2.$$

If we denote $\Psi(x) = \log \det(g)$, equation (\ref{Ric0}) is equivalent to the following equations
$$\frac{\partial^2 \Psi}{\partial z_\alpha \partial \bar z_\beta} = \Psi'' \bar z_\alpha z_\beta = 0 \ (\alpha \neq \beta), \ \ \frac{\partial^2 \Psi}{\partial z_\alpha \partial \bar z_\alpha} = \Psi' + \Psi'' |z_\alpha|^2 = 0,\ \alpha,\bar\beta=1, \dots, n.$$ 

This yields  $\Psi'=0$, i.e.
$$\log \det(g) = \log [(f')^{n-1}(f'+ f'' x)] = c,$$
for some constant $c$.

Setting $f'=y$ and $\tilde c = e^c > 0$, we get
$$y^{n-1}(y + xy') = y^n + x y' y^{n-1} = y^n + \frac{x}{n}(y^n)' = \tilde c$$
which rewrites as the following linear O.D.E. in $\xi = y^n$
$$\xi' = - \frac{n}{x} \xi + \tilde c \frac{n}{x}.$$
Therefore, one finds
$$y^n = \xi = C x^{-n} + \tilde c$$
 that is
$$f' = (C x^{-n} + \tilde c)^{\frac{1}{n}}$$
and then the general solution is

\begin{equation}\label{soluzGenint}
f(x) = \int (C x^{-n} + \tilde c)^{\frac{1}{n}} dt, \ C\in\R, \tilde c>0,
\end{equation}
which is equivalent to  (\ref{equationfint}) after a change of variables.
\end{proof}
\begin{remar}\rm
It is known that the metrics corresponding to the \K potentials (\ref{equationfint}) are non-complete and non-flat except in the case of the Euclidean metric ($\epsilon=0$).
\end{remar}
\begin{proof}[Proof of Theorem \ref{theomain}]
Let us denote by $\omega_\epsilon$ the \K form corresponding to the potential (\ref{equationfint}) with $\lambda =1$, namely 
\begin{equation}\label{omegaepsilon}
\omega_\epsilon =\frac{i}{2}\partial\bar\partial f_\epsilon,
\end{equation}
where
\begin{equation}\label{fepsilon}
f_\epsilon(x)=\int (\epsilon x^{-n} + 1)^{\frac{1}{n}} dx,\ \epsilon=-1, 0, 1.
\end{equation}
Notice that $\omega_{\epsilon}$ is flat either for $n=1$ or $\epsilon=0$.
We will show that for $n\geq 2$ we have the following:
\begin{itemize}
\item [(a)]
$\lambda\omega_{-1}$ is not projectively induced for any $\lambda\in\R^+$;
\item [(b)]
$\lambda \omega_1$ is not projectively induced for any $\lambda\in\R^+\setminus\Z$.
\end{itemize}
Then the proof of Theorem \ref{theomain} will follow by the very definition of stable-projectively induced metric.

A simple computation shows  that the function $g_3(x)$ (namely (\ref{gh}) for $h=3$) for the potential $f=\lambda f_{-1}$ is given by:
$$g_3(x)=\lambda\frac{(x^n-1)^{\frac{1-2n}{n}}}{x^3}\big(\lambda^2(x^n-1)^{\frac{2+2n}{n}}+3\lambda(x^n-1)^{\frac{1+n}{n}}-(x^n(n+1)-2)\big).$$
Hence, one has $\lim_{x \to 1^+}g_3(x)=-\infty$ and the proof of (a) follows by Corollary \ref{notprojind}.

In order to prove (b)   we first show by induction that  the function $g_h(x)$ for the potential $f=\lambda f_{1}$
is given by:
\begin{equation}\label{ghxnew}
g_h(x)=\frac{\lambda}{x^h}\Big(\Psi(x)\prod_{j=1}^{h-1}\big(\lambda\Psi(x)-j\big) +\varphi_h(x)x \Big),
\end{equation}
where $\Psi(x)=(x^{n}+1)^{1/n}$ and $\varphi_h\in C^\infty([0,+\infty))$. This statement is trivially true for $g_1$, because it is equal to $\frac{\lambda}{x}\Psi$. The functions $g_h$ can be defined recursively as 
$$g_{h+1}=g'_h+g_1g_h,$$ 
where $g_1=f'$.
Hence
$$g_{h+1}=
\frac{\lambda}{x^{h+1}}\big(\Psi\prod_{j=1}^{h}(\lambda\Psi-j) +\varphi_{h+1}x \big),$$
with 
$$\varphi_{h+1}=\frac{d}{dx}\big(\Psi\prod_{j=1}^{h-1}(\lambda\Psi-j)\big)+(1-h+\lambda\Psi)\varphi_h+\varphi_h'x\ \in C^\infty([0,+\infty))$$
and (\ref{ghxnew}) is proved.
Therefore, if $\lambda\in\R^+\setminus \Z$
$$\lim_{x \to 0^+}g_{[\lambda]+2}(x)= - \infty,$$
where $[\lambda]$ denotes the integral part of $\lambda$.
Thus,   Corollary \ref{notprojind} implies (b) and this concludes the proof of the theorem. 
\end{proof}
Notice that  we are able to  extend the proof of (b) also  for some fixed integer values of $\lambda$ with a case by case analysis. 
For example when $\lambda=1$ one obtains the following table which expresses  $g_h(x)$, for suitable values of $x$,  depending on the dimension $n$ of the domain, for $n=2, 3, 4, 5$:
\vskip 0.3cm

\begin{center}
\begin{tabular}{|c c c | r|}
\hline
$x$ & $h$ & $n$ & $g_h(x)$\\
\hline
$3/4$ & $7$ & $2$ & $-\frac{12294367331}{2373046875}$\\
$3/4$ & $5$ & $3 $& $\approx -2.81$\\
$3/4$ & $5$ & $4$ & $\approx -10.3$\\
$6/5$ & $4$ & $5$ & $\approx -0.14$\\
\hline
\end{tabular}
\end{center}

\vskip 0.3cm
Moreover, for any $n$, we have:
$$g_4(1)= 2^{\frac{1-3n}{n}}(8(2^{\frac{1}{n}})^3-24(2^{\frac{1}{n}})^2+30(2^{\frac{1}{n}})-15+8n2^{\frac{1}{n}}-9n)$$
which is seen to be negative for $n\geq 6$.
We believe (in accordance with Conjecture 1) that 
$\lambda \omega_1$ is not projectively induced  for all  integer values of $\lambda$ even if  we are not able to provide a general proof. 

Notice that for $n=2$ and $\epsilon =1$ one can explicitly express a \K potential for the \K  metric $\omega_1$ on $\C^2\setminus \{0\}$,
namely
\begin{equation}\label{EHpot}
f_1(x)=\sqrt{x^2+1}+\log x-\log (1+\sqrt{x^2+1}), \ x=|z_1|^2+|z_2|^2
\end{equation}
If  $M$ denotes the blow-up of $\mathds{C}^2$ at the origin and  $E$ denotes  the exceptional divisor
one can prove (see \cite{EH}) that there exists a complete Ricci-flat  and ALE \K  metric $g_{EH}$ on $M$
whose restriction to $\C^2\setminus \{0\}$ has \K potential given by (\ref{EHpot}).
This metric is known in the literature as the Eguchi--Hanson metric and denoted here by $g_{EH}$.

Therefore as a byproduct of our analysis one gets the following:
\begin{corol}
The Eguchi--Hanson metric $g_{EH}$ is not projectively induced.
\end{corol}

\begin{remar}\rm
Notice that if one will be able to prove that $\lambda g_{EH}$ is not projectively induced for all $\lambda>0$
(in accordance with our conjecture), this will provide an example of Ricci-flat and complete \K metric which does not admit a \K immersion into any finite or infinite dimensional  complex space form (the reader is referred to \cite{diastcigar} for details related to this issue).
\end{remar}

\section{Proof of Theorem \ref{theomain2} and Theorem \ref{theomain3}}
\begin{proof}[Proof of Theorem \ref{theomain2}]
By  (\ref{coefflu}) the assumption $a_3=0$ implies $\Delta |R|^2=0$. By 
a celebrated result of Yau \cite{Yau} (being $M$ complete) $(M, g)$ does not admit a nonconstant positive harmonic function.
Hence $|R|^2$ is constant. Being $g$ an ALE metric $|R|^2=0$ and so the metric $g$  is forced to be flat. This proves 1. 

In order to prove 2. it is enough to show that the vanishing of the term $a_3$ for the \K metric $g$ associated to the \K form  $\omega_\epsilon$ given by (\ref{omegaepsilon}) implies 
$g$ is flat, i.e. either $n=1$ or $\epsilon=0$.
Since
\begin{equation}\label{curvatureloc}
R_{i\bar j k \bar l}=\frac{\partial^2 g_{i\bar l}}{\partial z_k\partial\bar z_j}-\sum_{pq}g^{p\bar q}\frac{\partial g_{i\bar p}}{\partial z_k}\frac{\partial g_{q\bar l}}{\partial\bar z_j}.
\end{equation}
 one easily sees that the non-vanishing  components of the curvature tensor of the metric $g$  at $(z_1,0,\mathellipsis ,0)$ are:
\vspace{0.2 cm}\\
\hspace{0.2cm}
$\begin{array}{l}
R_{1\bar 1 1 \bar 1}=2f_\epsilon''+4f_\epsilon'''|z_1|^2+f_\epsilon''''|z_1|^4-\frac{1}{(f_\epsilon'+f_\epsilon''|z_1|^2)}(2f_\epsilon''+f_\epsilon'''|z_1|^2)^2|z_1|^2,\\
R_{1\bar 1 i \bar i}=f_\epsilon''+f_\epsilon'''|z_1|^2-\frac{1}{f_\epsilon'}(f_\epsilon'')^2  |z_1|^2,\\
R_{i\bar i i \bar i}=2R_{i\bar i j \bar j}=2f_\epsilon'',
\end{array}$\vspace{0.2 cm}\\
where $i,j\neq 1$ and $i\neq j$. 

Therefore, after a straightforward  but long computation,  taking into account the  curvature tensors symmetries
and the invariance of  $|R|^2$ under unitary transformations, we get
$$|R|^2=n(n-1)(n+1)(n+2)\epsilon^2(|z|^{2n}+\epsilon)^{-2(n+1)/n}.$$
Since 
$$\Delta |R|^2=g^{1\bar 1} (\frac{d|R|^2}{dx}+\frac{d^2|R|^2}{dx^2}x)+(n-1)g^{i\bar i} \frac{d}{dx}|R|^2, \ x=|z|^2$$
this yields (by Ricci-flatness)
$$a_3=\Delta |R|^2=2n(n-1)(n+2)(n+1)^2\epsilon^2((|z|^{2n}+\epsilon)^{-3(n+1)/n}(|z|^{2n}(n+3)-n\epsilon), $$
which vanishes either for $\epsilon =0$  or $n=1$.
\end{proof}

In order to prove Theorem \ref{theomain3} we recall the definition of Simanca's metric.

Let $M=\CP^2\sharp\overline{\CP^2}$ be the blow-up of $\mathds{C}^2$ at the origin and denote by $E$ the exceptional divisor.
Let $(z_1, z_2)$ be the standard coordinates of $\mathds{C}^2$. In \cite{simanca}  Simanca constructs a scalar flat K\"ahler complete (not Ricci-flat) metric  $g$ on $M$, 
whose K\"ahler potential  on $M\setminus E=\mathds{C}^2\setminus\{0\}$ can be written as 
\begin{equation}\label{potenzSimanca}
\Phi_S(|z|^2)=|z|^2+\log |z|^2.
\end{equation}

\begin{proof}[Proof of Theorem \ref{theomain3}]
The holomorphic map $$\varphi:\C^2\setminus \{ 0 \}\rightarrow \CP^{\infty}$$
given by
$$
(z_1,z_2)\mapsto (z_1,z_2,\mathellipsis ,\sqrt{\frac{j+k}{j!k!}}z_1^{j}z_2^{k},\mathellipsis), \ j+ k\neq 0,
$$
is a \K immersion from  $(\C^2\setminus \{ 0 \}, g_S)$ into $(\CP^{\infty}, g_{FS})$,
where $g_S$ denotes the restriction of the Simanca metric $g_S$ to $\C^2\setminus \{ 0 \}$ .
Indeed
$$\varphi^*\omega_{FS}=\frac{i}{2}\partial\bar\partial\log\sum_{j, k\in\N,\  j+k\neq 0}{\frac{j+k}{j!k!}}|z_1|^{2j}|z_2|^{2k}=\frac{i}{2}\partial\bar\partial\log (|z|^2e^{|z|^2})=
\frac{i}{2}\partial\bar\partial\Phi_S=\omega_S$$
Since $M=\CP^2\sharp\overline{\CP^2}$ is simply-connected it follows by Lemma \ref{calabi} that $\varphi$ extends to a \K immersion from $(M, g_S)$
into $(\CP^{\infty}, g_{FS})$.
It remains to  show that $a_3=0$.

By (\ref{coefflu}) and (\ref{values}) and taking into account that $a_2=0$ and  hence $|R|^2=4|Ric|^2$ (see \cite[Example 1]{LZhs}) one gets:
\begin{equation}\label{a3simanca}
a_3=\frac{1}{24}\big(-2 g^{\alpha\bar i}g^{j\bar \beta}g^{\gamma\bar k}Ric_{i\bar j,k}\overline{Ric_{\alpha\bar \beta,\gamma}}
+g^{\alpha \bar i}g^{j \bar \beta}g^{\gamma \bar k}g^{l\bar \delta}Ric_{i\bar j, k \bar l}R_{\beta\bar \alpha \delta \bar \gamma}-$$
$$-g^{\alpha \bar i}g^{j \bar \beta}g^{\gamma \bar k}g^{\delta\bar p}g^{q\bar\epsilon}R_{\beta\bar \gamma p \bar q}R_{k\bar \alpha \epsilon \bar \delta}Ric_{i\bar j}
-2g^{\alpha \bar i}g^{j \bar \beta}g^{\gamma \bar k}g^{ l\bar \delta}R_{i\bar j k \bar l}Ric_{\beta\bar \alpha}Ric_{\delta\bar \gamma}
\big).
\end{equation}
Since $a_3$ is invariant under unitary transformations, we only need to compute $a_3$ in $(z_1,0)$.
By (\ref{potenzSimanca}) we have

$$g=
\begin{pmatrix}
1 + \frac{|z_2|^2}{(|z_1|^2+|z_2|^2)^2} &  - \frac{z_2 \bar z_1}{(|z_1|^2+|z_2|^2)^2} \\
-\frac{z_1 \bar z_2}{(|z_1|^2+|z_2|^2)^2} & 1 + \frac{|z_1|^2}{(|z_1|^2+|z_2|^2)^2} \\
\end{pmatrix}$$

so that, for $z_2 = 0$,

$$g=
\begin{pmatrix}
1  & 0 \\
0 & \frac{|z_1|^2 + 1}{|z_1|^2} \\
\end{pmatrix}, \ \ g^{-1} =
\begin{pmatrix}
1  & 0 \\
0 & \frac{|z_1|^2}{|z_1|^2 + 1} \\
\end{pmatrix} .$$

Combining this with (\ref{curvatureloc}) we deduce that  the unique components  different  from zero when evaluated at $(z_1,0)$ are:\vspace{0.2 cm}\\
\hspace{0.2cm}
$\begin{array}{l}
R_{1\bar 1 1 \bar 1}=2\Phi_S''+4\Phi_S'''|z_1|^2+\Phi_S''''|z_1|^4-\frac{1}{(\Phi_S'+\Phi_S''|z_1|^2)}(2\Phi_S''+\Phi_S'''|z_1|^2)^2|z_1|^2=0\\
R_{1\bar 1 2 \bar 2}=\Phi_S''+\Phi_S'''|z_1|^2-\frac{1}{\Phi_S'}(\Phi_S'')^2  |z_1|^2=\frac{1}{|z_1|^2(|z_1|^2+1)}\\
R_{2\bar 2 2 \bar 2}=2\Phi_S''=-\frac{2}{|z_1|^4}
\end{array}$\vspace{0.2 cm}\\
By recalling that $Ric_{i\bar j}=-\frac{\partial^2 \log\det g}{\partial z_i\partial\bar z_j}$ one gets:
$$Ric=
\begin{pmatrix}
-\frac{1}{(|z_1|^2+1)^2} &0\\
0& \frac{1}{(|z_1|^2+1)|z_1|^2}\\
\end{pmatrix}$$
By definition $Ric_{i\bar j,k}=\partial_k Ric_{i\bar j}- Ric_{p\bar j}\Gamma_{ki}^p$, where $\Gamma_{ki}^p$ are Christoffel's symbols, given  by $\Gamma_{ki}^p = g^{p \bar q} \frac{\partial g_{i \bar q}}{\partial z_k}$.

 A straightforward computation gives that the unique first covariant derivatives different from zero are
\vspace{0.2 cm}\\
\hspace{0.2cm}
$\begin{array}{l}
Ric_{1\bar 1,1}=\frac{2}{(|z_1|^2+1)^3}\bar z_1\\
Ric_{2\bar 2,1}=Ric_{1\bar 2,2}=-\frac{2}{|z_1|^2(|z_1|^2+1)^2}\bar z_1
\end{array}$\vspace{0.2 cm}\\
Finally, we compute only the following second covariant derivatives (by definition $Ric_{i\bar j,k\bar l}=\partial_{\bar l}\partial_k Ric_{i\bar j}+\Gamma_{ki}^q\Gamma_{\bar l\bar j}^{\bar p}Ric_{q\bar p}- \Gamma_{ki}^p\partial_{\bar l}Ric_{p\bar j} -\partial_{\bar l}\Gamma_{ki}^p Ric_{p\bar j} -\Gamma_{\bar l\bar j}^{\bar p}\partial_kRic_{i\bar p}  $).
\vspace{0.2 cm}\\
\hspace{0.2cm}
$\begin{array}{l}
Ric_{1\bar 1,2\bar 2}=Ric_{2\bar 1,1\bar 2}=\frac{4(4|z_1|^2-1)}{|z_1|^2(|z_1|^2+1)^6}-\frac{1}{|z_1|^2(|z_1|^2+1)^3}\\
Ric_{2\bar 2,1\bar 1}= Ric_{1\bar 2,2\bar 1}=\frac{4(4|z_1|^2-1)}{|z_1|^2(|z_1|^2+1)^6}+\frac{1}{|z_1|^2(|z_1|^2+1)^3} \\
Ric_{2\bar 2,2\bar 2}=-\frac{4}{|z_1|^2(|z_1|^2+1)^2}
\end{array}$\vspace{0.2 cm}\\
Substituting  in (\ref{a3simanca}), after a long but straightforward computation  one gets $a_3=0$, 
and we are done.
\end{proof}

\end{document}